\documentclass[11pt,reqno]{amsart}
\usepackage{amssymb,amsfonts,amsmath,geometry,amsthm}
\newtheorem{theorem}{Theorem}[section]
\newtheorem{lemma}[theorem]{Lemma}
\newtheorem{prop}[theorem]{Proposition}

\theoremstyle{definition}
\newtheorem{defn}[theorem]{Definition}

\newtheorem{example}[theorem]{Example}
\pdfoutput=1
\title{Meromorphic Convexity on Stein manifolds}
\author{Blake J. Boudreaux}
\address[Blake J. Boudreaux]{The University of Western Ontario}
\email{bboudre7@uwo.ca}
\author{Rasul Shafikov}
\address[Rasul Shafikov]{The University of Western Ontario}
\email{shafikov@uwo.ca; \rm{ This author is is partially supported by NSERC.}}
\date{}
\begin{document}
\maketitle
\begin{abstract}
We consider generalizations of rational convexity to Stein manifolds and prove related results.
\end{abstract}

Given a compact set $K$ in $\mathbb C^n$, its rationally convex hull is defined as
\begin{equation}\label{e.rat}
\mathcal{R}\text{-hull }K= \{ z\in \mathbb C^n: |R(z)| \le \|R\|_K, \text{ for all rational functions } R \text{ with poles off } K\}.
\end{equation}
A compact $K$ is called {\it rationally convex} if $K =\mathcal{R}\text{-hull }K$. This is a standard definition of convexity with respect to a family of functions $\mathcal F$, when $\mathcal F$ is chosen to be the family of rational functions. $K$ is rationally convex if and only if for any point $p \in \mathbb C^n \setminus K$ there exists a complex algebraic hypersurface that passes through $p$ but avoids $K$.

In this paper we prove new results concerning generalizations of rational convexity to Stein manifolds. In this setting the defining family
$\mathcal F$  can be chosen to be meromorphic or strongly meromorphic functions. As it turns out, this corresponds to convexity with respect to complex hypersurfaces or principal hypersurfaces.
This distinction leads to different notions of convexity, which we call {\it meromorphic} and {\it strong meromorphic} convexity. In Theorem~\ref{t.approx} we give general approximation results corresponding to (strong) meromorphic convexity.

While determining whether a given compact is rationally or meromorphically convex may be a difficult problem, there exists a characterization by Duval--Sibony~\cite{DuSi} of rational convexity of a special, but important class of compacts, namely totally real manifolds. This characterization establishes a strong connection between rational convexity and K\"ahler geometry. In Section~\ref{s.duvsib} we give a generalization of this result for strong meromorphic convexity. In Section~\ref{s.strong} we give a sufficient and necessary conditions for a meromorphically convex compact and totally real manifolds to be strongly meromorphically convex. Section 4 generalizes this circle of ideas to subsemigroups $G$ of the Picard group $\text{Pic}(X)$ by showing that a holomorphic function defined on a neighbourhood of a $G$-meromorphically convex compact (see Definition~\ref{G-cvx}) can be approximated uniformly on $K$ by quotients of the ``strong'' form $s_1/s_2$, where $s_1,s_2\in\Gamma(X,L)$ for some $L\in G$.

\section{Meromorphic convexity}\label{s.merc}

It is easy to see  that given a compact $K\subset \mathbb C^n$ its rationally convex hull $\mathcal{R}\text{-hull }K$, as defined
by~\eqref{e.rat}, consists of points $z\in \mathbb C^n$ with the property that if $f(z)=0$ for some polynomial $f$, then $f$ vanishes
somewhere on $K$. This means that
the complement of $\mathcal{R}\text{-hull }K$ is a union of algebraic hypersurfaces. Note that any complex hypersurface in $\mathbb C^n$ is principal, i.e., is the zero locus of a single entire function (holomorphic on $\mathbb C^n$). Any holomorphic function on a rationally convex $K$ can be approximated uniformly on $K$ by rational functions, according to the classical Oka--Weil theorem.

Let now $X$ be a Stein manifold. A natural generalization of rational convexity to $X$
is to replace the family of rational functions with meromorphic functions on $X$ or, as considered by many authors, to replace complex algebraic hypersurfaces simply with
complex hypersurfaces. Let us start with meromorphic functions. A general meromorphic function can be defined as follows.
Let $\mathcal M_p$ be the quotient field of $\mathcal O_p$, i.e., $\mathcal M_p$ is the field of germs of meromorphic functions
at a point $p$ of a complex manifold $X$.
A meromorphic function $m$ on $X$ is a map
\[
m: X \to \cup_{p\in X} \mathcal M_p, \ p \to \mathcal M_p,
\]
such that $m_p \in \mathcal M_p$ for all $p \in X$, and for every $p\in X$ there exists a connected neighbourhood  $U\subset X$ and holomorphic functions $f,g \in \mathcal O(U)$, $g\not\equiv 0$, such that $m_z=f_z/g_z$ for all $z\in U$. The quotient of two entire functions is clearly a meromorphic function, and on Stein manifolds one may construct global meromorphic functions from compatible local data by solving an additive Cousin problem. A point $p \in X$ is called a point of indeterminacy of a meromorphic function $m$ if $m_p = f_p/g_p$, the germs $f_p$ and $g_p$ are coprime, and $f(p) = g(p)=0$. The set of all indeterminacy points is called the indeterminacy locus of the meromorphic function $m$ and will be denoted by ${\mathcal I}(m)$.

The Poincar\'e problem asks whether every global meromorphic function on $X$ is the quotient of entire functions on $X$. The \textit{strong} Poincar\'e problem requires that the entire functions of the quotient satisfy an additional property: their germs are relatively prime at every point of $X$. We will call the latter type of meromorphic function a \textit{strong meromorphic function}.
We denote by  $\mathcal{M}(X)$ and $\mathcal{SM}(X)$ the spaces of meromorphic and strong meromorphic functions on $X$, respectively.

On a Stein manifold $X$ the solution to the (weak) Poincar\'e problem  is an immediate consequence of Cartan's Theorem A. The strong Poincar\'e problem is however solvable precisely when the topological condition $H^2(X,\mathbb{Z})=0$ is satisfied. Since $H^2(X,\mathbb{Z})=0$ implies the universal solvability of the multiplicative Cousin problem on Stein manifolds (they are in fact equivalent), its sufficiency for the solvability of the strong Poincar\'e problem is clear. On the other hand, its necessity cannot be immediately be concluded from classical results, and was shown by Ephraim~\cite{Ep} to be a consequence of Cartan's Theorem B. For a thorough treatment of this subject, see Fritzsche--Grauert~\cite{FrGr}.

Consider also the situation where every (weak) meromorphic function on a Stein manifold $X$ has zero divisor corresponding to a torsion element of $\text{Pic}(X)\cong H^2(X,\mathbb{Z})$. Then for every $m\in\mathcal{M}(X)$ with zero divisor $D$ there is a positive integer $k$ and $h\in\mathcal{O}(X)$ for which $\text{div}(h)=kD$, so $h/m^k\in\mathcal{O}(X)$ and hence one sees that $m^k\in\mathcal{SM}(X)$ via the representation $m^k=h/(h/m^k)$. Therefore, in such a situation the hulls defined below will coincide despite $H^2(X,\mathbb{Z})\neq 0$ in general.

\begin{defn}\label{d.mercon}
Let $X$ be a Stein manifold and $K$ be a compact subset of $X$. Define
\begin{align*}
	\mathcal{M}\text{-hull}(K)&=\left\{p\in X\,:\,\left|m(p)\right|\leq \left\|m\right\|_K\text{ for all }m\in\
	\mathcal{M}(X)\text{ with }{\mathcal I}(m)\cap (K\cup\{p\})=\varnothing\right\} ,\\
	\mathcal{SM}\text{-hull}(K)&=\left\{p\in X\,:\,\left|m(p)\right|\leq \left\|m\right\|_K\text{ for all }m\in
	\mathcal{SM}(X)\text{ with }{\mathcal I}(m)\cap (K\cup\{p\})=\varnothing\right\} .
\end{align*}
We call $K$ {\it meromorphically convex} (resp. {\it strongly meromorphically convex}) if $\mathcal{M}\text{-hull}(K) = K$ (resp.
$\mathcal{SM}\text{-hull}(K) = K$).
\end{defn}
For brevity we also write $\mathcal{M}$-convex (resp. $\mathcal{SM}$-convex) for meromorphic (resp. strong meromorphic) convexity. Clearly,
\begin{equation}\label{e.minsm}
 K \subset  \mathcal{M}\text{-hull}(K) \subset \mathcal{SM}\text{-hull}(K) .
\end{equation}
 We will see that for a compact $K$ on a Stein manifold $X$, the sets $\mathcal{M}\text{-hull}(K)$ and $\mathcal{SM}\text{-hull}(K)$ are also compact. It follows then from the definition that $\mathcal{M}\text{-hull}(K)=\mathcal{M}\text{-hull}(\mathcal{M}\text{-hull}(K))$, and similarly for the strong meromorphically convex hull. If $X= \mathbb C^n$, then both definitions agree with rational convexity because entire functions on $\mathbb C^n$ can be approximated by polynomials uniformly on compact sets.

As in the case of rational convexity, $\mathcal{M}$- and $\mathcal{SM}$-convexity can be also formulated in terms of complex hypersurfaces. What differentiates the two definitions is whether one requires the hypersurfaces to be principal. More precisely,  let $K$ be a compact subset of $X$ and let us define the following hulls:
\begin{align*}
	H(K)&=\left\{x\in X : f^{-1}(0)\cap K\neq\varnothing\text{ for every }f\in\mathcal{O}(X)\text{ satisfying }f(x)=0\right\},\\
	h(K)&=\left\{x\in X : \text{every complex hypersurface in } X \text{ passing through }x\text{ intersects }K\right\}.
\end{align*}
Clearly $h(K)\subseteq H(K)$, and it was shown by Col\c{t}oiu~\cite{Co} that these hulls coincide for all $K \subset X$ if and only if $\text{Hom}(H_2(X,\mathbb{Z}),\mathbb{Z})=0$, a slightly weaker condition than $H^2(X,\mathbb{Z})=0$. The next result shows that the hulls with respect to the families of meromorphic functions and hypersurfaces coincide.

\begin{prop}\label{p.smero}
Let $X$ be a Stein manifold and $K$ be a compact subset of $X$. Then
\begin{equation}
	h(K) =\mathcal{M}\text{-hull}(K), \ \ H(K) =\mathcal{SM}\text{-hull}(K).
\end{equation}
\end{prop}

In view of the proposition meromorphic convexity can also be called {\it convexity with respect to hypersurfaces}.

\begin{proof}
Throughout the proof it will be convenient to view a meromorphic function $m$ on $X$  as a holomorphic map
from $X\setminus {\mathcal I}(m)$ into~$\mathbb{CP}^1$.

Suppose that $p\not\in H(K)$. Then there exists a $f\in\mathcal{O}(X)$ such that $f(p)=0$ and whose zero locus avoids
$K$. It follows that $1/f\in\mathcal{SM}(X)$ with $\infty = 1/|f(p)|>\|1/u\|_K$. This shows that
$(X\setminus H(K)) \subset (X \setminus \mathcal{SM}\text{-hull}(K) )$.

For the opposite inclusion, suppose $p\in X \setminus \mathcal{SM}\text{-hull}(K)$, so there exists a strong meromorphic
function $f/g$ with ${\mathcal I}(f/g)\cap (K\cup\{p\})=\varnothing$ satisfying $|f(p)/g(p)|>\|f/g\|_K$.
Set $z:=f(p)/g(p)$.  If $z=\infty$, then since $f$ and $g$ are coprime, $g$ is the entire function whose zero locus passes
through $p$ but avoids $K$, and so $p \in X\setminus H(K)$. If $z\ne\infty$, then the holomorphic function
$f-z\cdot g$ vanishes at $p$. If $f(q)-z\cdot g(q)=0$ for some $q\in K$, then either $f(q)=g(q)=0$ or $f(q)/g(q)=z$ must hold.
Since $f$ and $g$ are relatively prime at $q$, the former case implies that $q\in {\mathcal I}(f/g)$, which is not possible.
The latter case contradicts the inequality $|f(p)/g(p)|>\|f/g\|_K$. Therefore, such a $q\in K$ does not exist and
hence $p\not\in H(K)$. This completes the proof of the first equality of the proposition.

\bigskip

For the proof of the second identity we use a lemma by Ephraim~\cite{Ep} below, along with a small modification of the proof, so for convenience we provide these below.

\begin{lemma}\label{l.ephraim}
Let $X$ be a Stein manifold, $\{H_j\}_{j \in J}$ be a locally finite family of irreducible complex hypersurfaces in $X$,
and $\nu_j \ge 0$ be integers. Then there exists an entire function $f$ that vanishes on $H_j$ to order $\nu_j$ for $j\in J$.
\end{lemma}

\begin{proof}
Let $\mathcal J$ be the sheaf of germs of holomorphic functions vanishing on each $H_j$ to order at least $\nu_j$. For each $j\in J$,  choose a point $p_j \in H_j$. Then $P=\{p_j, j\in J\}$ is a complex analytic subset of $X$ of dimension 0. Let $\mathfrak{n}$ be the sheaf of germs of holomorphic functions vanishing on $P$---a coherent sheaf of ideals.

If $p \notin P$, then $\mathcal J_p = (\mathfrak{n} \cdot\mathcal J)_p$. On the other hand, by Nakayama's lemma,
for $p\in P$ we have $\mathcal J_p \ne (\mathfrak{n}\cdot \mathcal J)_p$. It follows that $(\mathcal J/\mathfrak{n} \cdot \mathcal J)_p \ne 0$ iff $p\in P$. Since $P$ is discrete, we may find a section $s \in \Gamma(X, \mathcal J/\mathfrak{n} \cdot \mathcal J)$ for which $s_p \ne 0$ for any $p \in P$.

By Cartan's Theorem B there exists a section $f\in \Gamma(X,\mathcal J)$ whose image in $\Gamma (X, \mathcal J/\mathfrak{n} \cdot \mathcal J)$ is precisely $s$. This means that $f$ is a holomorphic function which vanishes to order at least $\nu_j$ on $H_j$ for all $j \in J$. But since $s_p \ne 0$ for all $p\in P$, it follows that
$f_p \notin (\mathfrak{n} \cdot \mathcal J)_p$ for all $p \in P$. Taking $p=p_j$ we see that $f$ vanishes to order at most $\nu_j$ on $H_j$.
\end{proof}

\noindent Note that the function $f$ constructed in the lemma may vanish somewhere outside
$\cup_j H_j$.

\bigskip

Returning to the proof of the proposition, suppose that $q\not\in h(K)$. Then there exists a hypersurface $Z$ passing through $q$ but avoiding $K$. Without loss of generality, we can assume that $Z$ is irreducible. By Lemma~\ref{l.ephraim}, there exists an $f\in\mathcal{O}(X)$ whose zero locus contains $Z$. Further, we can choose $q$ as one of the members of the discrete set $P$ in the proof of the lemma to ensure that there is a neighbourhood $U$ of $q$ for which $U\cap f^{-1}(0)= U\cap Z$.  Write $f^{-1}(0)=Z\cup E$, where $E$ is a hypersurface in $X$ that necessarily avoids $U$.
	We apply Lemma~\ref{l.ephraim} again to find another $g\in\mathcal{O}(X)$ such that $E\subset g^{-1}(0)$ and $g$ has multiplicity zero along $Z$. By a similar modification of the proof of the lemma we can ensure that $g\neq 0$ near $q$. Therefore $f/g\in\mathcal{M}(X)$, $q\not\in {\mathcal I}(f/g)$, and the zero divisor of $f/g$ and $Z$ agree (as sets). It follows that $\infty = |g(q)/f(q)|>\|g/f\|_K$ since $f/g$ has no zeroes on $K$. This shows that $(X\setminus h(K))\subset (X\setminus\mathcal{M}\text{-hull}(K))$.

The reverse inclusion is a modification of a standard argument. Suppose $q\in X$ is point for which there exists a function  $f/g\in\mathcal{M}(X)$ with ${\mathcal I}(f/g)\cap (K\cup\{q\})=\varnothing$ and satisfying
$|f(q)/g(q)|>\|f/g\|_K$. Set $z:=f(q)/g(q)$. If $z=\infty$, then the zero divisor of $g/f$ is the desired hypersurface. If $z \ne \infty$, then
\[
	\left\|\frac{f}{g}-z\right\|_K\geq |z|-\left\|\frac{f}{g}\right\|_K>|z|-\left|\frac{f(q)}{g(q)}\right|=0.
\]
Since this inequality is strict, the zero divisor of the meromorphic function $f/g-z$ is the required hypersurface.
\end{proof}

\begin{prop}\label{p.rat}
Let  $X$ be a Stein manifold and $\Phi : X \to \mathbb C^N$ be a proper holomorphic embedding.
A compact $K\subset X$ is $\mathcal{SM}$-convex if and only if $\Phi(K)$ is rationally convex in $\mathbb C^N$.
\end{prop}

\begin{proof}
Suppose $K$ is $\mathcal{SM}$-convex.
Applying Cartan's Theorem A to the sheaf of germs of holomorphic functions on $\mathbb{C}^N$ vanishing on the complex analytic set $\Phi(X)$ yields global generators $h_1,\ldots,h_M\in\mathcal{O}(\mathbb{C}^N)$ with the property that $\bigcap_{j=1}^{M}h^{-1}_j(0)=\Phi(X)$.
	Accordingly, fix $a\in\mathbb{C}^N\setminus\Phi(K)$. If $a\not\in\Phi(X)$, then $h_j(a)\neq 0$ for one of the generators $h_1,\ldots,h_M$, in which case $z\mapsto h_j(z)-h_j(a)$ is a holomorphic function on $\mathbb{C}^N$ which passes through $a$ but avoids $\Phi(K)$. Approximation by Taylor polynomials provides a polynomial with the same property. If $a\in \Phi(X)\setminus\Phi(K)$, then the strong meromorphic convexity of $S$ (and Proposition~\ref{p.smero}) yields a $f\in\mathcal{O}(X)$ with $f(\Phi^{-1}(a))=0$ but $f^{-1}(0)\cap K=\varnothing$. Applying the Oka--Cartan extension theorem to $f\circ\Phi^{-1}\in\mathcal{O}(\Phi(X))$ provides an entire function $F$ with the property that $F(a)=0$ but $F^{-1}(0)\cap\Phi(K)=\varnothing$. This shows that $\Phi(K)$ is rationally convex.

Conversely, if $\Phi(K)$ is rationally convex, then for any $a\in \Phi(X) \setminus \Phi(K)$ there exists a holomorphic polynomial on
$\mathbb C^N$ that vanishes at $a$ but does not vanish on $\Phi(K)$. Its restriction to $\Phi(X)$ is an entire function that defines a principal hypersurface through $a$ that avoids $\Phi(K)$. It follows from Proposition~\ref{p.smero} that $K$ is $\mathcal{SM}$-convex.
\end{proof}

The rationally convex hull of any compact in $\mathbb C^N$ is compact, and so it follows from Proposition~\ref{p.rat} that $\mathcal{SM}\text{-hull}(K)$ is compact for any compact $K \subset X$. If $p\notin \mathcal{M}\text{-hull}(K)$ and $h$ is a hypersurface on $X$ passing through $p$ and avoiding $K$, then $h$ can be perturbed (see Ex.~\ref{e.app} for details) so that it passes through any point in a small neighbourhood of $p$ in $X$ still avoiding $K$. This shows that $\mathcal{M}\text{-hull}(K)$ is a closed set, and so in view of~\eqref{e.minsm}, it is compact. A different proof of this is given in~\cite{Hi}.

\section{Meromorphic approximation}\label{s.approx}

We now discuss approximation on (strongly) meromorphically convex compacts.
Hirschowitz~\cite[Thm~2]{Hi} already proved a version of the Oka--Weil theorem for compacts on a Stein manifold that are meromorphically convex: \textit{any holomorphic function defined on a neighbourhood of a compact subset $K$ of a Stein manifold with $h(K)=K$ is the uniform limit on $K$ of a sequence of meromorphic functions without poles on $K$}. It may appear that this
notion of convexity is the proper analogue of rational convexity to Stein manifolds. However, our next result shows that the desired notion of convexity depends on the class of meromorphic functions by which one wishes to approximate. More precisely, the following holds.

\begin{theorem}\label{t.approx}
	Let $X$ be a Stein manifold and $K$ be a strongly meromorphically convex subset of~$X$. For any
	$\varphi\in\mathcal{O}(K)$ and $\varepsilon>0$ there exist $f,g\in\mathcal{O}(X)$ that are pointwise relatively prime at each point of $X$ and satisfy $\|\varphi-f/g\|<\varepsilon$.
\end{theorem}

\noindent Note the conclusion $\|\varphi-f/g\|_K<\varepsilon$ implies in particular that $f/g$ has no poles on $K$.

\begin{proof}[Proof of Theorem~\ref{t.approx}]
Our proof uses the methods of Hirschowitz~\cite{Hi}. Let $K$ be a meromorphically compact subset of $X$ and let $f$ be a
holomorphic function defined on a neighbourhood $U$ of~$K$. Define the compact set $L:=\widehat{K}_X\setminus U$,
where $\widehat{K}_X$ denotes the holomorphically convex hull of $K$ in $X$. If $H$ is a principal hypersurface in $X$, then $X \setminus H$ is Stein, and if $H$ avoids $K$, then $\widehat K_{X \setminus H}$ is compact in $X\setminus H$. Therefore, the set $L_H = L \cap \widehat K_{X \setminus H}$ is compact, and $\bigcap_{K \cap H = \varnothing} L_{H} =\varnothing$.
Since $L$ is compact, we can find finitely many principal hypersurfaces $H_1,\ldots,H_k$ avoiding~$K$ and
\[
	L\cap \widehat{K}_{X\setminus H_1}\cap\cdots\cap\widehat{K}_{X\setminus H_k}=\varnothing .
\]
Writing $H:=\bigcup_{j=1}^{k}H_j$ we see that $\widehat{K}_{X\setminus H}\subset U$, and hence by the classical Oka--Weil theorem~\cite[Theorem 18]{FoFoWo} we may approximate $f$ uniformly on $\widehat{K}_{X\setminus H}$ by members of $\mathcal{O}(X\setminus H)$.

It suffices now to approximate functions in $\mathcal{O}(X\setminus H)$ by strongly meromorphic functions.
To do this, choose a proper holomorphic embedding $\Phi:X\to\mathbb{C}^N$ for some large $N$, which exists,
since $X$ is Stein.  The hypersurface $H$ is principal, since the $H_j$ are, and so
$H=h^{-1}(0)$ for some $h\in\mathcal{O}(X)$. Then $\Psi:=(\Phi,h)$ embeds $X\setminus H$ into $\mathbb{C}^{N+1}\setminus \{z_{N+1}=0\}$. Let $f \in \mathcal{O}(X\setminus H)$. The Oka--Cartan theorem allows us to extend the function
$f\circ\Psi^{-1}: \Psi(X\setminus H) \to \mathbb C$ to a function
$F:\mathbb{C}^{N+1}\setminus\{z_{N+1}=0\} \to \mathbb C$, which in turn may be approximated uniformly on compacts by partial sums of its Laurent series expansion with respect to $z_{N+1}$,
\[
	F(z',z_{N+1})=\sum_{k\in\mathbb{Z}}a_k(z')z_{N+1}^{k},
\]
where the $a_k$ are entire functions of the first $N$ variables---in fact by taking an appropriate Taylor polynomial it can be assumed that the $a_k$ are polynomials. Taking a partial sum of the above series and precomposing with $\Psi$ yields normal approximation of $f$ by meromorphic functions of the form
	\begin{align*}
	\sum_{k=-m}^ma_k(\Phi)h^k&=\frac{a_{-m}(\Phi)}{h^m}+\frac{a_{-m+1}(\Phi)}{h^{m-1}}+\ldots+a_m(\Phi)h^m\\
	&=\frac{a_{-m}(\Phi)+a_{-m+1}(\Phi)h+\ldots+a_m(\Phi)h^{2m}}{h^m}.
	\end{align*}
	The meromorphic function above is strong if the polynomial $a_{-m}$ is not identically zero on any irreducible component of the complex-analytic set $\Phi(h)$. Since the set of polynomials in $\mathbb{C}^N$ satisfying this property is dense in the space of all polynomials, taking a small perturbation of $a_{-m}$ if necessary ensures that the approximating meromorphic functions are strong.
\end{proof}

A natural question is whether approximation by strong meromorphic functions is possible on compacts that are only meromorphically convex. The following example shows this is not possible in general.

\begin{example}\label{e.app}
Let $X$ be a Stein manifold for which there exists a compact $K$ satisfying
	\[\mathcal{M}\text{-hull}(K)\neq \mathcal{SM}\text{-hull}(K).\]
Such a manifold is known to exist due to the work of Col\c{t}oiu~\cite{Co}. Fix a point $p\in \mathcal{SM}\text{-hull}(K)\setminus \mathcal{M}\text{-hull}(K)$.
We claim that $\widetilde K:=\mathcal{M}\text{-hull}(K)\cup\{p\}$ is meromorphically convex. Indeed, let $q\not\in\widetilde K$.
Then there exists a hypersurface $Z$ which passes through $q$ but avoids $K$.
Suppose $Z$ passes through $p$. There is a small perturbation of $Z$ that passes through $q$ but avoids $\widetilde K$. Indeed,
let $L$ be the holomorphic line bundle corresponding to $Z$, then there exists a holomorphic section $s: X \to L$ whose zero
locus is precisely $Z$. It is well-known that any line bundle $L$ admits a bundle embedding into the trivial bundle
$X \times \mathbb C^N$
for some $N>0$ (see, e.g., \cite[Cor. 7.3.2]{Fo}). Treating $s$ as a map into $\mathbb C^N$ we may find a map
$\widetilde s: X \to \mathbb C^N$, which is a small perturbation of $s$ and $\widetilde s(q) \ne 0$ and $\widetilde s^{-1}(0) \cap K = \varnothing$. Then the projection of $\widetilde s$ to $L$ is a holomorphic section of $L$ whose zero locus avoids $\widetilde K$.
This shows that  $\mathcal{M}\text{-hull}(\widetilde K)=\widetilde K$ as claimed.

Suppose that any holomorphic function defined on a neighbourhood of $\widetilde K$ is the uniform limit on $\widetilde K$ of a sequence of strong meromorphic functions. Consider the function
\[
\psi(z)=
\begin{cases}
	0,&\text{ when }z\in K,\\
	1,&\text{ when }z=p .
\end{cases}
\]
By trivial extension, we may consider $\psi$ as a holomorphic function defined in a neighbourhood of~$\widetilde K$. By our assumption, given $\varepsilon>0$ there exists a $f/g\in\mathcal{SM}(X)$ without poles on $\widetilde K$ such that $\|\psi-(f/g)\|_{\widetilde K}<\varepsilon$, in particular,  $|1-f(p)/g(p)|<\varepsilon$. On the other hand, since
$p \in \mathcal{SM}\text{-hull}(\widetilde K)$ we have
\[
	\left|\frac{f(p)}{g(p)}\right|\leq\left\|\frac{f}{g}\right\|_K<\varepsilon.
\]
This is a contradiction for small  $\epsilon>0$.
\end{example}

In the case of $X=\mathbb{C}^n$, a partial converse to Oka--Weil is known~\cite[Thm 1.2.10]{St}: \textit{A compact $K\subset\mathbb{C}^n$ is rationally convex whenever any continuous function on $K$ is the uniform limit on $K$ of rational functions with poles off $K$}. We will now show the analogues for strong and weak meromorphic convexity hold. This is where the equivalence of meromorphic and hypersurface convexity plays an important role.

For a compact $K\subset X$ denote by $\mathcal{M}(K)$ (resp. $\mathcal{SM}(K)$) the uniformly closed subalgebra of $\mathcal{C}(K)$ that consists of all the functions that can be approximated uniformly on $K$ by meromorphic (resp. strongly meromorphic) functions with poles off $K$.

\begin{theorem}
	A compact subset $K$ of a Stein manifold $X$ is meromorphically convex if $\mathcal{M}(K)=\mathcal{C}(K)$. It is strongly meromorphically convex if $\mathcal{SM}(K)=\mathcal{C}(K)$.
\end{theorem}
\begin{proof}
	We follow Stout~\cite{St}. First note that every $m\in\mathcal{M}(K)$ has a natural extension to the compact set $\mathcal{M}$-hull$(K)$. Indeed, if $\{m_k\}_{k=1}^{\infty}$ is a sequence of meromorphic functions with poles off $K$ converging to $m$ uniformly on $K$, then $\{m_k\}_{k=1}^{\infty}$ is in fact Cauchy and hence convergent at any point $a\in\mathcal{M}$-hull$(K)$, since $|m_k(a)-m_\ell(a)|\leq \|m_k-m_\ell\|_K$ for $k,\ell\in\mathbb{N}$. We denote by $\widehat{m}\in\mathcal{M}(\mathcal{M}$-hull$(K))$ this extension of $m$ to $\mathcal{M}$-hull$(K)$. This yields a natural identification of $\mathcal{M}\big(\mathcal{M}\text{-hull}(K)\big)$ with $\mathcal{M}(K)$.

	Seeking a contradiction, suppose $K$ is not meromorphically convex and choose $a\in\mathcal{M}\text{-hull}(K)\setminus K$. Then the $\mathbb{C}$-linear functional $T$ on $\mathcal{M}(K)$ defined by $m\mapsto \widehat{m}(a)$ satisfies $T(u\cdot v)=T(u)T(v)$ for each $u,v\in\mathcal{M}(K)$. Since $\mathcal{M}(K)=\mathcal{C}(K)$ by assumption, $T$ is also a $\mathbb{C}$-linear functional on $\mathcal{C}(K)$ satisfying $T(f\cdot g)=T(f)T(g)$ for all $f,g\in\mathcal{C}(K)$. It is well-known~\cite[Thm 1.2.8]{St} that all such functionals on $\mathcal{C}(K)$ can be realized as evaluation functionals at a unique point of $K$, i.e., there exists a unique point $z\in K$ such that $T(f)=f(z)$ for all $f\in\mathcal{C}(K)$. But this means that
\[
	T(m)=m(a)=m(z)
\]
	for all $m\in\mathcal{M}\big(\mathcal{M}\text{-hull}(K)\big)$, and since the algebra $\mathcal{M}\big(\mathcal{M}\text{-hull}(K)\big)$ separates points ($X$ is Stein), this is a contradiction.

	The case of $\mathcal{C}(K)=\mathcal{SM}(K)$ is an identical argument.
\end{proof}

\section{Duval--Sibony for Strong Meromorphic Convexity}\label{s.duvsib}
Recall that a submanifold $S$ of a complex manifold $X$ is \textit{totally real} if for every $x\in S$ the tangent space $T_xS$ contains no complex line.
Duval and Sibony~\cite[Theorem 3.1]{DuSi} proved the following striking result:  {\it a smooth compact totally real submanifold
$S \subset\mathbb{C}^n$  is rationally convex if and only if there exists a smooth strictly plurisubharmonic function $\varphi$ on $\mathbb{C}^n$ such that $\iota^*_S\text{d}\text{d}^c\varphi=0$, where $\iota_S:S\to\mathbb{C}^n$ is the inclusion map.} Note that  $\omega:=\text{d}\text{d}^c\varphi$ a K\"ahler form on $\mathbb C^n$.

We say that a K\"ahler form $\omega$ on a complex manifold $X$ is a \textit{Hodge form} if $[\omega]\in H^2(X,\mathbb{Z})$, that is, $[\omega]\in H_{\text{dR}}^2(X,\mathbb{R})$ lies in the image of the morphism $H^2(X,\mathbb{Z})\to H^2(X,\mathbb{R})\cong H^2_{\text{dR}}(X,\mathbb{R})$ induced by the containment $\mathbb{Z}\hookrightarrow\mathbb{R}$. Guedj~\cite{Gu} further generalized the theorem of Duval--Sibony to the context of complex projective manifolds and Stein manifolds. The generalization to  Stein manifolds can be stated using the terminology of the present work as follows.

\begin{theorem}[Guedj~\cite{Gu}, Thm 5.8]\label{t.DSG}
Let $S$ be a smooth compact totally real submanifold of a Stein manifold $X$. the following are equivalent:
\begin{enumerate}
	\item[(i)] $S$ is meromorphically convex.
	\item[(ii)] There exists a smooth Hodge form $\omega$ for $X$ such that $\iota^*_S\omega=0$.
\end{enumerate}
\end{theorem}

A natural question in the context of this note is whether such a characterization of \textit{strongly} meromorphic compact totally real manifolds exists. We have the following.

\begin{theorem}\label{t.DSBS}
Let $S$ be a smooth compact totally real submanifold of a Stein manifold $X$. The following are equivalent:
\begin{enumerate}
	\item[(i)] $S$ is strongly meromorphically convex.
	\item[(ii)] There exists a a smooth strictly plurisubharmonic function $\varphi$ on $X$ such that $\iota^*_S\text{d}\text{d}^c\varphi=0$.
\end{enumerate}
\end{theorem}

Observe that a K\"ahler form $\omega$ has a $\text{d}\text{d}^c$-potential---as in condition (ii) of Theorem~\ref{t.DSBS}---if and only if $[\omega]=0\in H^2(X,\mathbb{Z})$. Indeed, it is clear that $[\text{d}\text{d}^c\varphi]=0\in H^2(X,\mathbb{Z})$ whenever $\varphi\in\text{PSH}(X)$. Conversely, if $\omega$ is a Hodge form, then there exists a holomorphic line bundle $L\to X$ and metric $\psi$ on $L$ such that $\text{d}\text{d}^c\psi=\omega$  \cite[Thm 13.9(b)]{De}. If we know further that $[\omega]=0\in H^2(X,\mathbb{Z})$, then $L$ has Chern class zero and hence is isomorphic to the trivial bundle on $X$. It follows that $\psi$ can be realized as a global strongly plurisubharmonic function on $X$. This is the case, for example, when $X = \mathbb C^n$, in this setting Theorems~\ref{t.DSG} and~\ref{t.DSBS} are both reduced precisely to the statement of Duval--Sibony.

Our proof of Theorem~\ref{t.DSBS} is done through the embedding of the Stein manifold $X$ into a complex Euclidean space,
which allows us to circumvent the more sophisticated methods of Guedj~\cite{Gu} and use the methods of Duval--Sibony~\cite{DuSi} directly. We require a variation of some known results on compact K\"ahler manifolds, these are formulated in the lemma below,
the proof of which will be given at the end of this section.

\begin{lemma}[cf. {\cite[Prop 2.1]{CoGuZe}}, {\cite[Thm 4.1]{OrVe}}, {\cite[Thm 4]{Sc}}]\label{l.DSBS}
Let $X$ be a Stein manifold and $Z\subset X$ be a closed complex submanifold equipped with a Hodge form $\omega$. If there exists a K\"ahler form $\eta$ on $X$ with $[\iota^*_Z\eta]=[\omega]\in H^2(Z,\mathbb{Z})$, then $\omega$ admits an extension off of any prescribed relatively compact subset of $Z$ to a K\"ahler form on $X$; that is, for any compact $B\subset\subset Z$ there exist a K\"ahler form $\widetilde\omega$ on $X$ such that $\iota^*_B\widetilde\omega=\omega$.
\end{lemma}

\begin{proof}[Proof of Theorem~\ref{t.DSBS}]
First, we properly embed $X$ into some Euclidean space $\mathbb{C}^N$ via the holomorphic mapping $\Phi: X\to\mathbb{C}^N$.

Suppose $S$ is $\mathcal{SM}$-convex. Then by Proposition~\ref{p.rat}, $\Phi(S)$ is rationally convex in $\mathbb{C}^N$.
In view of the result of Duval--Sibony cited at the beginning of the section, there exists a $\psi\in\text{PSH}(\mathbb{C}^N)$ with $\iota^*_{\Phi(S)}\text{d}\text{d}^c\psi=0$,
and hence $\varphi:=\psi\circ\Phi\in\text{PSH}(X)$ satisfies $\iota^*_S\text{d}\text{d}^c\varphi=0$.

Conversely, suppose there exists a smooth strongly plurisubharmonic function $\varphi$ on $X$ such that $\iota^*_S\text{d}\text{d}^c\varphi=0$. Since $[\text{d}\text{d}^c(\varphi\circ\Phi^{-1})]=[\iota^*_X\text{d}\text{d}^c(|\cdot|^2)]=0\in H^2(X,\mathbb{Z})$,
by Lemma~\ref{l.DSBS}
$\text{d}\text{d}^c(\varphi\circ\Phi^{-1})$ admits an extension $\omega$ off of some large ball containing $\Phi(S)$ to all of $\mathbb{C}^N$ as a K\"ahler form. Because $\mathbb{C}^N$ is topologically trivial, there exists a strictly plurisubharmonic $\psi$ on $\mathbb{C}^N$ with $\text{d}\text{d}^c\psi=\omega$. Since
\[
	\iota^*_{\Phi(S)}\text{d}\text{d}^c\psi=\iota^*_{\Phi(S)}\text{d}\text{d}^c\omega=\iota^*_S\text{d}\text{d}^c\varphi=0,
\]
applying Duval--Sibony in the other direction shows that $\Phi(S)$ is rationally convex. It follows from Proposition~\ref{p.rat} that
$S$ is strongly meromorphically convex.
\end{proof}

\begin{proof}[Proof of Lemma~\ref{l.DSBS}]
	Let $\psi\in\mathcal{C}^\infty(X)$ be a strictly plurisubharmonic exhaustion function for $X$. Without loss of generality we can assume $B=B_1\cap Z$, where $B_1=\{z\in X\,:\,\psi(z)<c_1\}$ for some $c_1>0$. Choose $c_3>c_2>c_1$ so that their respective sublevel sets $B_3:=\{\psi<c_3\}$ and $B_2:=\{\psi<c_2\}$ satisfy $B_1\subset\subset B_2\subset\subset B_3\subset\subset X$

	Since $[\iota^*_Z]=[\omega]$ and $\omega$ is in particular K\"ahler, there exists a $\varphi\in\mathcal{C}^{\infty}(Z)$ and an $\varepsilon>0$ such that
\[
	\iota^*_Z\eta+\text{d}\text{d}^c\varphi=\omega\geq\varepsilon\cdot\iota^*_Z\eta.
\]
	We now proceed as in Coman--Guedj--Zeriahi~\cite{CoGuZe}: choose $\varphi_1$ to be any extension of $\varphi$ to $X$ and define
\[
	\varphi_2=\varphi_1+A\chi\text{dist}(\,\cdot\,,Z)^2,
\]
	where $\chi\in\mathcal{C}^{\infty}(X)$ is a cutoff function supported in a small neighbourhood of $Z$ that is identically one near $Z$, and $A>0$. Here the distance function can be any Riemannian distance on $X$, e.g., the distance associated to the K\"ahler metric $\eta$. Now $\varphi_2$ is another smooth extension of $\varphi$ to $X$, and by choosing $A$ large enough we can ensure that
\[
	\eta+\text{d}\text{d}^c\varphi_2\geq\frac{\varepsilon}2\eta\quad\text{on }B_3.
\]
	Define $u=\chi\log (\text{dist}(\,\cdot\,,X)^2)$; by shrinking the support of $\chi$ (and consequently increasing $A>0$ if necessary), we can ensure that the function $\log (\text{dist}(\,\cdot\,,X))^2$ is well-defined and quasi-plurisubharmonic on $\text{supp}(\chi)$. Hence there is a small $\delta>0$ such that $\delta\cdot\text{d}\text{d}^cu\geq-\eta$ on $B_3$.

	We define one more smooth extension of $\varphi$:
\[
	\varphi_3=\frac{1}{2}\log\left(e^{2\varphi_2}+e^{\delta u+C}\right).
\]
A standard calculation at points of $B_3$ yields
\begin{align*}
	\eta+\text{d}\text{d}^c\varphi_3&\geq\eta+\frac{2e^{2\varphi_2}\text{d}\text{d}^c\varphi_2+\delta e^{\delta u+C}\text{d}\text{d}^c u}{2(e^{2\varphi_2}+e^{\delta u+C})}\\
	&=\frac{2e^{2\varphi_2}(\eta+\text{d}\text{d}^c\varphi_2)+e^{\delta u+C}(\eta+\delta\text{d}\text{d}^cu)}{2(e^{2\varphi_2}+e^{\delta u+C})}\\
	&\geq\frac{\varepsilon e^{2\varphi_2}}{2(e^{2\varphi_2}+e^{\delta u+C})}\eta\geq\frac{\varepsilon}4\eta.
\end{align*}
	It follows that $\eta+\text{d}\text{d}^c\varphi_3$ is a K\"ahler form on $B_3$ which extends $\omega$ off $B$.

	To complete the proof we will modify this form off of $B_1$ so that it is K\"ahler on all of $X$, through a standard procedure. Choose a smooth function $h:\mathbb{R}\to [0,\infty)$ that is constant for $t\leq c_1$, is strictly convex and increasing for $t\in (c_1,c_2)$, and $h(t)=t$ for $t\geq c_3$. Then $h\circ\psi$ is plurisubharmonic on $X$, strictly plurisubharmonic outside $\overline{B}_1$, and vanishes on $\overline{B}_1$. Next, choose a smooth function $\lambda:\mathbb{R}\to [0,1]$ that is identically one for $t\leq c_2$ and identically zero for $t>c_3$. The form
\[
	\widetilde\omega:=\eta+\text{d}\text{d}^c\big((\lambda\circ\psi)\varphi_3\big)+C'\cdot\text{d}\text{d}^c(h\circ\psi)
\]
extends $\iota^*_B\omega$, and is K\"ahler for $C'>0$ large enough.
\end{proof}

\section{Characterization of Strong Meromorphic Convexity}\label{s.strong}
As mentioned above, Col\c{t}oiu~\cite{Co} showed that if a Stein manifold $X$ satisfies $\text{Hom}\big(H_2(X;\mathbb{Z});\mathbb{Z}\big)\ne 0$, then there exist compacts $K\subset X$ for which $\mathcal{SM}\text{-hull}(K)\ne \mathcal{M}\text{-hull}(K)$. In this context a natural question is the following: {\it when do $\mathcal{M}\text{-hull}(K)$ and $\mathcal{SM}\text{-hull}(K)$ coincide for a \textup{given} compact~$K$?}
To formulate our results we first introduce some terminology.
\begin{defn}
A domain $\Omega$ on a manifold $X$ is called meromorphically (resp. strongly meromorphically) Runge if $\mathcal{M}\text{-hull}(K)$ (resp. $\mathcal{SM}\text{-hull}(K)$) with respect to $X$ is compact in $\Omega$ for every compact subset $K$ of $\Omega$.
\end{defn}

We first consider the case when $K=S$ is a meromorphically convex totally real manifold on~$X$. Recall that by Theorem~\ref{t.DSG} there exists a smooth Hodge form $\omega$ on $X$ such that $\iota^*_S \omega =0$. Further, there exists a holomorphic line bundle $L\to X$ and metric $\varphi$ on $L$ such that $\iota^*_{S}\text{d}\text{d}^c\varphi=0$.

\begin{theorem}\label{t.BSC}
	Let $S$ be a smooth compact totally real submanifold of a Stein manifold $X$ that is meromorphically convex.
	Then $S$ is strongly meromorphically convex if and only if there exists a Stein neighbourhood $U$ of $S$ that is strongly
meromorphically Runge and an integer $k>0$ such that $L^{\otimes k}|_U$ is trivial.
\end{theorem}

\begin{proof}
Suppose there exists a Stein neighbourhood $U$ of $S$ which is strongly meromorphically Runge and an integer $k$ such that
the line bundle $L^{\otimes k}|_U$ is trivial. Then $k\varphi|_U$ is a strictly plurisubharmonic function, so by Theorem~\ref{t.DSBS} we see that
$S$ is convex with respect to strong meromorphic functions on $U$. Since $U$ is strongly meromorphically Runge, we at least have
$\mathcal{SM}\text{-hull}(S)\subset U$. We claim that to prove that $S=\mathcal{SM}\text{-hull}(S)$ it suffices to find for every $a\in U\setminus S$ a function $g\in\mathcal{O}(X)$ whose zero locus passes through $a$ but avoids $S$. Indeed, if this holds, then the compact $\mathcal{SM}\text{-hull}(S)$ contains $S$ as a connected component. A simple argument using Theorem~\ref{t.approx} shows that a connected component of a $\mathcal{SM}$-convex compact is $\mathcal{SM}$-convex, and this proves the claim.

Accordingly, fix $a\in U$. Then there exists a $f\in\mathcal{O}(U)$ such that $f(a)=0$ and whose zero set avoids $S$. Since $U$ is strongly meromorphically Runge, $f$ can be approximated normally in $U$ by members of $\mathcal{SM}(X)$ with poles outside a large compact of $U$ (Theorem~\ref{t.approx}), so for $u/v\in\mathcal{SM}(X)$ sufficiently close to $f$ on a neighbourhood of $S\cup\{a\}$, the function $z\mapsto u(z)-v(z)\tfrac{u(a)}{v(a)}$ is holomorphic on $X$ with a zero at $z=a$ and has zero locus omitting $S$. This proves that $S$ is strongly meromorphically convex.

	Conversely, suppose that $S={\mathcal SM}\text{-hull}(S)$. Since $S$ is totally real, $\text{dist}^2(x,S)$---the square-distance function to $S$---is strictly plurisubharmonic in a small tubular neighbourhood $U$ of $S$. Set $\rho:=\text{dist}^2(\cdot,S)|_U$ and note that $S=\{x\in U\,:\,\rho(x)\leq 0\}$. Because $S={\mathcal SM}\text{-hull}(S)$, Boudreaux--Gupta--Shafikov~\cite[Thm 1.2]{BoGuSh} shows
that there exists an extension of $\text{d}\text{d}^c\rho$ to a K\"ahler form $\omega$ on $X$ with $[\omega]=[0]\in H^2(X,\mathbb{Z})$. Applying the converse implication of the same theorem, we see that neighbourhoods of the form $U_{\varepsilon}:=\{x\in U\,:\,\rho(x)<\varepsilon\}$ are Stein and have closures that are convex with respect to strong meromorphic functions as well.
It follows that $U_\varepsilon$ is strongly meromorphically Runge for $\varepsilon>0$ small enough.

	We lastly must show that $L^{\otimes k}|_{U_\varepsilon}$, the line bundle $L\to X$ given to us by Theorem~\ref{t.DSG} (see the paragraph above the statement of Theorem~\ref{t.BSC}), is trivial. Fix such a small $\varepsilon>0$. By Sard's theorem, we may assume that $U_{\varepsilon}$ has smooth boundary and hence has finitely generated cohomology groups. We will show that $L^{\otimes k}|_{U_\varepsilon}$ is trivial for some positive integer $k$. The neighbourhood $U_{\varepsilon}$ is a deformation retract of $S$, and so in particular we have $H^2_{\text{dR}}(S,\mathbb{R})\cong H^2_{\text{dR}}(U_\varepsilon,\mathbb{R})$. But
\[
	c_1(L|_S)=[\iota^*_S\text{d}\text{d}^c\varphi]=[0]\in H^2_{\text{dR}}(S,\mathbb{R}) ,
\]
	so $c_1(L|_{U_\varepsilon})=[0]\in H^2_{\text{dR}}(U_\varepsilon,\mathbb{R})$ as well; that is, the image of the first Chern class of $L$ in $H^2_{\text{dR}}(U_\varepsilon,\mathbb{R})\cong H^2(U_\varepsilon,\mathbb{R})$ through the morphism $H^2(U_\varepsilon,\mathbb{Z})\to H^2(U_\varepsilon,\mathbb{R})$ induced by the containment $\mathbb{Z}\hookrightarrow\mathbb{R}$ is zero. Since the kernel of this morphism is precisely the torsion subgroup of $H^2(U_\varepsilon,\mathbb{Z})$, there exists an integer $k$ so that $c_1(L^k|_{U_\varepsilon})=[0]\in H^2(U_\varepsilon,\mathbb{Z})$. Furthermore, $U_\varepsilon$ is Stein, so the operator $c_1:\text{Pic}(U_\varepsilon)\to H^2(U_\varepsilon,\mathbb{Z})$ is an isomorphism, and we can conclude that $L^k|_{U_{\varepsilon}}$ is trivial.
\end{proof}

Next we formulate a characterization of strong meromorphic convexity for arbitrary compacts.

\begin{theorem}
	Let $K$ be a meromorphically convex compact in a Stein manifold $X$. Then $K$ is strongly meromorphically convex if and only if $K$ admits a strongly meromorphically Runge Stein neighbourhood $U$ with the property that for every $a\in U\setminus K$ there exists a line bundle $L$ in the torsion subgroup of $\text{Pic}(U)$ and a section $\sigma\in\Gamma(U,L)$ such that $\sigma(a)=0$ but $\sigma^{-1}(0)\cap K=\varnothing$.
\end{theorem}

\begin{proof}
	Suppose that $K$ admits such a neighbourhood. Since $U$ is strongly meromorphically Stein, we at least have $\mathcal{SM}\text{-hull} (K)\subset U$ and so it suffices to show that for every $a\in U\setminus K$ there exists a $f\in\mathcal{O}(X)$ with $f(a)=0$ but $f^{-1}(0)\cap K$. Given $a\in U\setminus K$, by assumption we know there is a line bundle $L$ in the torsion subgroup of $\text{Pic}(U)$ and a section $\sigma\in\Gamma(U,L)$ such that $\sigma(a)=0$ and $\sigma^{-1}(0)\cap K=\varnothing$. This means there is an integer $k$ such that $\sigma^k\in\mathcal{O}(U)$ that can then be approximated normally on $U$ by members of $\mathcal{SM}(X)$ with poles outside of some large compact in $U$. We can now proceed as in the proof of the previous theorem: choosing a $u/v\in\mathcal{SM}(X)$ that approximates $\sigma^k$ close enough on a neighbourhood of the compact $K\cup\{a\}$, we see that $z\mapsto u(z)-v(z)\tfrac{u(a)}{v(a)}$ is a member of $\mathcal{O}(X)$ with the zero set that passes through $a$ but avoids $K$.

	The converse is trivial: If $K$ is strongly meromorphically convex, then every neighbourhood of $K$ has this property, since the trivial line bundle will then satisfy the hypotheses. So choose $U$ to be any (possibly large) strongly meromorphically Runge neighbourhood of $K$.
\end{proof}


\section{Further Generalizations}

Upon noticing that every hypersurface of $X$ can be realized as the zero set of a global holomorphic section of some holomorphic line bundle $L\to X$, one might be drawn to consider a notion of convexity with respect to global holomorphic sections of the \textit{fixed} line bundle $L$. However, if the zero set of a section $s\in\Gamma(X,L)$ avoids a compact $K\subset X$, then so does the zero set of the section $s^M\in\Gamma(X,L^{\otimes M})$ for any positive integer $M$. Therefore, a more appropriate notion of convexity of this type is to consider convexity with respect to the subgroup $\langle L\rangle\leq\text{Pic}(X)$ generated by $L$. Here we use the notation $``G_1\leq G_2$'' to indicate $G_1$ is a sub(semi)group of $G_2$. Recall that a semigroup $G$ is nonempty set equipped with an associative binary operation.

In this vein, Abe~\cite{Ab} considered more generally convexity with respect to subsemigroups of $\text{Pic}(X)$, defined as follows.
\begin{defn}\label{G-cvx}
Let $X$ be a Stein manifold containing a compact $K$ and $G$ be a subsemigroup of $\text{Pic}(X)$. Define
\[
	G\text{-hull}(K)=\left\{x\in X\,:\, s^{-1}(0)\cap K\neq\varnothing\text{ for every }L\in G\text{ and }s\in\Gamma(X,L)\text{ satisfying }s(x)=0\right\}.
\]
	We call $K$ \textit{meromorphically convex with respect to $G$}, or simply, \textit{$G$-meromorphically convex}, if $G\text{-hull}(K)=K$.
\end{defn}
It is clear that $G_2\text{-hull}(K)\subseteq G_1\text{-hull}(K)$ whenever $G_1\leq G_2$. Furthermore, it is known that $G\text{-hull}(K)$ is compact in $X$ for any $G\le\text{Pic}(X)$, see~\cite[Prop 4.1 and Cor 4.5]{Ab}. It is immediate from the definitions that $\langle 1\rangle\text{-hull}(K)=H(K)$ and $\text{Pic}(X)\text{-hull}(K)=h(K)$, where $1\in\text{Pic}(X)$ denotes the trivial line bundle.

Convexity with respect to a subsemigroup $G\leq\text{Pic}(X)$ is of interest in view of the following generalization of the Oka--Weil Theorem~\cite[Thm 5.1]{Ab}: \textit{If $K$ is a compact subset of a Stein manifold with $G\textit{-hull}(K)=K$, then for every $f\in\mathcal{O}(K)$ and $\varepsilon>0$ there exist $L\in G$ and $s_1,s_2\in\Gamma(X,L)$ such that $\|f-s_1/s_2\|_K<\varepsilon$.} However, this statement alone undervalues Abe's work. Indeed, a compact $K$ with $G\text{-hull}(K)=K$ is in particular meromorphically convex, so as mentioned above~\cite[Thm 2]{Hi} any $f\in\mathcal{O}(K)$ can be approximated uniformly on $K$ by members of $\mathcal{M}(X)$ with poles off $K$. Furthermore, no control over the $L\in G$ from which the approximating meromorphic functions $s_1/s_2$ are built is given, so one is left to wonder why approximation by quotients of sections of $L$ would be preferred over approximation by quotients of holomorphic functions (which are sections of the trivial bundle).

For a given $L\in\text{Pic}(X)$, we say that two sections $s_1,s_2\in\Gamma(X,L)$ are \textit{coprime} if their zero loci share no irreducible components. With this in mind, we give the following strengthening of  Abe's result.

\begin{theorem}\label{t.BSa}
	Let $X$ be a Stein manifold and $G$ a subsemigroup of $\text{Pic}(X)$. Let $K$ be a compact set of $X$ such that $G\text{-hull}(K)=K$. Then for every $f\in\mathcal{O}(K)$ and for every $\varepsilon>0$ there exist $L\in G$ and coprime $s_1,s_2\in\Gamma(X,L)$ such that $\|f-s_1/s_2\|_K<\varepsilon$.
\end{theorem}
\noindent\textbf{Remarks.} (i) Observe that meromorphic functions of the form $s_1/s_2$ for coprime $s_1,s_2\in\Gamma(X,L)$ are natural generalizations of strong meromorphic functions to sections of a line bundle. Furthermore, note that every $m\in\mathcal{M}(X)$ can be written in strong form with respect to some line bundle. Indeed, the zero divisor of $m$ can be realized as $\text{div}(s)$ for some global section $s$ of some $L\in\text{Pic}(X)$. Write $s\cong\{s_i\}_{i\in I}$, where $\{U_i\}_{i\in I}$ is an open cover of $X$ by trivializations of $L$ with associated transition functions $g_{ij}\in\mathcal{O}^*(U_i\cap U_j)$. Then we have
\[
	\frac{s_i}{m|_{U_i}}=g_{ij}\frac{s_j}{m|_{U_j}}
\]
on $U_{i}\cap U_j$, and it follows that $\{s_i/m|_{U_i}\}_{i\in I}$ patches together to form a global section of $L$. Hence $m=s/(s/m)$ is a strong representation of $m$ in terms of sections of $L$.

(ii) Since (weak) meromorphic convexity is equivalent to convexity with respect to the entire Picard group, a consequence of Theorem~\ref{t.BSa} is the following:\textit{ Any holomorphic function defined in a neighbourhood of a meromorphically convex compact $K\subset X$ is the uniform limit on $K$ of a sequence of meromorphic functions having the form $s_1/s_2$ for \textup{coprime} $s_1,s_2\in\Gamma(X,L)$, where $L\in\text{Pic}(X)$ only depends on $f$.} This may have utility in situations where it is more desirable that the meromorphic functions by which one wishes to approximate be quotients of objects which are coprime rather than be quotients of functions themselves.

\begin{proof}[Proof of Theorem~\ref{t.BSa}]
	Let $U$ be an open subset of $K$ on which $f$ is defined. Define the compact set $M:=\widehat{K}_X\setminus U$, where $\widehat{K}_X$ denotes the holomorphically convex hull of $K$ in $X$. As in the proof of Theorem~\ref{t.approx}, for any $s\in\Gamma(X,L)$, $L\in G$, with zero locus avoiding $K$ the set $X\setminus s^{-1}(0)$ is Stein; in particular, $\widehat{K}_{X\setminus s^{-1}(0)}\subset X\setminus s^{-1}(0)$. The set $M_s:=M\cap\widehat{K}_{X\setminus s^{-1}(0)}$ is compact, and so $\bigcap_{s}M_s=\varnothing$, where the intersection is taken over all such $s$, implies there are finitely many $s_j\in\Gamma(X,L_j)$, $L_j\in G$, $j=1,\ldots,k$, with
\[
	M\cap\widehat{K}_{X\setminus s_1^{-1}(0)}\cap\cdots\cap\widehat{K}_{X\setminus s_k^{-1}(0)}=\varnothing.
\]
	Consequently, the section $s:=\prod_{j=1}^{k}s_j$ is a global holomorphic section of $L_1\otimes\ldots\otimes L_k\in G$ which has $\widehat{K}_{X\setminus s^{-1}(0)}\subset U$ and hence $f$ may be approximated uniformly on $\widehat{K}_{X\setminus H}$ by members of $\mathcal{O}(X\setminus s^{-1}(0))$.

	To complete the proof, it suffices to show the following: Let $X$ be a Stein manifold and $s\in\Gamma(X,L)$ for some holomorphic line bundle $L$. Then every $f\in\mathcal{O}(X\setminus s^{-1}(0))$ can be approximated normally by meromorphic functions of the form $\sigma/s^N$, where the zero set of $\sigma\in\Gamma(X,L^{\otimes N})$ shares no irreducible components with $s^{-1}(0)$. Define $\mathcal{J}=(1/s)\cdot\mathcal{O}$, where $\mathcal{O}$ denotes the sheaf of germs of holomorphic functions on $X$. Then $\mathcal{J}$ is a coherent subsheaf of germs of meromorphic functions on $X$~\cite[pg. 119]{GrRe}, and so by Cartan's theorem A there exist global sections $h_1,\ldots,h_k$ of the sheaf $\mathcal{J}$ with the property that the germs $(h_1)_z,\ldots,(h_k)_z$ generate $\mathcal{J}_z$ as an $\mathcal{O}_z$-module for every $z\in X$. Given a proper holomorphic embedding $\Phi:X\to\mathbb{C}^N$, we define $\Psi=(\Phi,h_1,\ldots,h_k)$. Observe that $\Psi$ is a proper holomorphic embedding of $X\setminus s^{-1}(0)$ into $\mathbb{C}^{N+k}$. Indeed, if $p\in s^{-1}(0)$, then $(\frac{1}{s})_p=(g_1)_p(h_1)_p+\ldots+{(g_k)}_p{(h_k)}_p$ for some ${(g_1)}_p,\ldots,{(g_k)}_p\in\mathcal{O}_p$. Since $1/s\to\infty$ along any sequence in $X\setminus s^{-1}(0)$ tending towards $p$, the same must be true for at least one of the $h_j$.

	Now, the Oka--Cartan theorem yields a function $F\in\mathcal{O}(\mathbb{C}^{N+k})$ which agrees with $f\circ\Psi^{-1}$ when restricted to the complex-analytic set $\Psi(X\setminus s^{-1}(0))\subset\mathbb{C}^{N+k}$. Let $F_T$ be a Taylor polynomial of $F$ and consider $F_T\circ\Psi$. Let $\{U_i\}$ be an open cover of $X$ by trivializations of $L$ with associated transition functions $g_{ij}$. Since the $h_1,\ldots,h_k$ are in particular sections of $\mathcal{J}$, we have $h_\nu|_{U_i}=t_{i\nu}/s_i$ in $U_i$, $t_{i\nu}\in\mathcal{O}(U_i)$, for each $\nu$ as well; consequently $F_T\circ\Psi|_{U_i}$ is a polynomial in $t_{i1}/s_i,t_{i2}/s_i,\ldots,t_{ik}/s_i$ and the components of $\Phi$ restricted to $U_i$. Putting everything under one denominator shows $F_T\circ\Psi|_{U_i}$ to be of the form
\[
	\frac{a_{-N}(t_i,\Phi)+a_{-N+1}(t_i,\Phi)s_i+\ldots+a_{N}(t_i,\Phi)s_i^{2N}}{s_i^N},
\]
	where $a_{-N},\ldots,a_{N}$ are polynomials and $t_i$ denotes $(t_{i1},\ldots,t_{ik})$. Note that $N$ is independent of the choice of $i$. Write $\sigma_i$ for the numerator of the above expression. We have $\sigma_i=(s_i/s_j)^N\sigma_j=g_{ij}^N\sigma_j$ on $U_i\cap U_j$ for every $i,j$, so $\sigma=\{\sigma_i\}$ patches together to a section of $L^{\otimes N}$. Similarly to Theorem~\ref{t.approx}, $\sigma$ and $s$ can be assumed coprime after possibly a small perturbation of $a_{-N}$.
\end{proof}


\end{document}